\documentclass{amsart}
\usepackage{amsmath,amsthm,amssymb}
\usepackage{graphicx}
\usepackage{tikz-cd}
\usepackage{paralist}
\usepackage{environ}
\usepackage{xcolor}
\usepackage{hyperref}
\usepackage{centernot}
\usepackage{mathtools}

\usepackage{marvosym}
\usepackage{xcolor,cancel}

\usepackage{pstricks}

\makeatletter
\def\square{\pst@object{square}}
\def\square@i(#1,#2)#3{{\use@par\solid@star\psframe[origin={#1,#2}](#3,#3)}}
\makeatother

\makeatletter
\DeclareFontFamily{U}{tipa}{}
\DeclareFontShape{U}{tipa}{bx}{n}{<->tipabx10}{}
\newcommand{\arc@char}{{\usefont{U}{tipa}{bx}{n}\symbol{62}}}%

\newcommand{\arc}[1]{\mathpalette\arc@arc{#1}}

\newcommand{\arc@arc}[2]{%
  \sbox0{$\m@th#1#2$}%
  \vbox{
    \hbox{\resizebox{\wd0}{\height}{\arc@char}}
    \nointerlineskip
    \box0
  }%
}
\makeatother

\makeatletter
\newcommand{\doublewedge}{\big@doubleop{\wedge}}
\newcommand{\big@doubleop}[1]{%
  \DOTSB\mathop{\mathpalette\big@doubleop@aux{#1}}\slimits@
}

\newcommand\big@doubleop@aux[2]{%
  \sbox\z@{$\m@th#1#2$}%
  \makebox[1.35\wd\z@][s]{$\m@th#1#2\hss#2$}%
}

\newcommand{\abs}[1]{\left|#1\right|}     
\newcommand{\cl}{\mbox{c$\ell$}}  
\newcommand{\Int}{\mbox{int}} 
\newcommand{\bdy}{\mbox{bdy}} 
\newcommand{\sh}{\mbox{sh}}
\newcommand{\cyc}{\mbox{cyc}} 
\newcommand{\hcyc}{\mbox{hCyc}} 
\newcommand{\Hcyc}{\mbox{HCyc}} 
\newcommand{\hSys}{\mbox{hSys}} 

\newcommand{\near}{\delta} 
\newcommand{\dcap}{\mathop{\cap}\limits_{\Phi}} 
\newcommand{\dnear}{\delta_{\Phi}} 
\newcommand{\norm}[1]{\left\|#1\right\|}  

\newcommand{\Hquad}{\hspace{0.25em}} 

\newcommand{\Hb}{\mbox{Hb}} 

\usepackage{pstricks,pst-text,pst-grad,pst-node,pst-3dplot,pstricks-add,pst-poly,pst-coil,pst-bar,pst-all,pst-plot,pst-2dplot} 
\usepackage{pst-fun,pst-blur}
\usepackage{pst-all}
\usepackage{subfigure}
\renewcommand{\thesubfigure}{\thefigure.\arabic{subfigure}}
\makeatletter
\renewcommand{\p@subfigure}{}
\renewcommand{\@thesubfigure}{\thesubfigure:\hskip\subfiglabelskip}
\makeatother

\theoremstyle{plain}
\newtheorem{theorem}{Theorem}
\newtheorem{lemma}{Lemma}
\newtheorem{remark}{Remark}
\newtheorem{definition}{Definition}

\newtheorem{example}{Example}
\newtheorem{corollary}{Corollary}
\newtheorem{proposition}{Proposition}

\usepackage[title]{appendix}

\begin{document}

\title{Good Coverings of Proximal Alexandrov Spaces.
Path Cycles in the Extension of the Mitsuishi-Yamaguchi good covering and Jordan Curve Theorems.}

\author[J.F. Peters]{J.F. Peters}
\address{
Computational Intelligence Laboratory,
University of Manitoba, WPG, MB, R3T 5V6, Canada and
Department of Mathematics, Faculty of Arts and Sciences, Ad\.{i}yaman University, 02040 Ad\.{i}yaman, Turkey,
}
\email{james.peters3@umanitoba.ca}
\thanks{The research has been supported by the Natural Sciences \&
Engineering Research Council of Canada (NSERC) discovery grant 185986 
and Instituto Nazionale di Alta Matematica (INdAM) Francesco Severi, Gruppo Nazionale per le Strutture Algebriche, Geometriche e Loro Applicazioni grant 9 920160 000362, n.prot U 2016/000036 and Scientific and Technological Research Council of Turkey (T\"{U}B\.{I}TAK) Scientific Human
Resources Development (BIDEB) under grant no: 2221-1059B211301223.}
\author[T. Vergili]{T. Vergili}
\address{
Department of Mathematics, Karadeniz Technical University, Trabzon, Turkey,
}
\email{tane.vergili@ktu.edu.tr}

\subjclass[2010]{54E05; 55P57)}

\date{}

\dedicatory{Dedicated to Camille Jordan and S.A. Naimpally}

\begin{abstract}
This paper introduces proximal path cycles, which lead to the main results in this paper, namely, extensions of the Mitsuishi-Yamaguchi Good Coverning Theorem with
different forms of Tanaka good cover of an Alexandrov space equipped with a proximity relation as well as extension of the Jordan curve theorem.  In this work, a {\bf path cycle} is a sequence of maps $h_1,\dots,h_i,\dots,h_{n-1}\mbox{mod}\ n$ in which $h_i:[0,1]\to X$ and $h_i(1) = h_{i+1}(0)$ provide the structure of a path-connected cycle that has no end path.   An application of these results is also given for the persistence of proximal video frame shapes that appear in path cycles. 
\end{abstract}
\keywords{Cycle, Good Cover, Homotopy, Nerve, Path, Proximity.}
\maketitle
\tableofcontents

\section{Introduction}
A bounded surface region results from a path-connected sequence of paths that has no end path.  A {\bf homotopic path} (briefly, {\bf path}) is a continuous map from the unit interval into a space.  For a space $S$, $h:[0,1]\to S$ can be either a straight line from a point $h(0)\in S$ to a surface point $h(1)\in S$ or surface curved line~\cite{Puisseux1850algFns} or a cross cut (also called an {\bf ideal arc}~\cite[\S 3, p.11]{Mosher1988surface}), which is a line that punctures a surface boundary at $h(0)$, passing through the interior of the surface without self-intersections and exiting at a point $h(1)$ on the surface boundary~\cite{Eynde1992paths},~\cite{Mosher1988surface}.  The focus here is on finite, bounded, simply-connected surfaces.  A planar surface $S$ is {\bf simply connected}, provided every path $h$ lies entirely on $S$, {\em i.e.}, every path has all its interior points $h(t)\in \Int(S), t\in (0,1)$ and its end points $h(0),h(1)\in S$ and $h$ has no self-loops.  

Paths are stitched together in a sequence ({\em aka} path cycle, also called a train track~\cite[p.53]{Mosher1988surface}) to delineate a bounded surface region.  A {\bf path cycle} $E$ (denoted by $\hcyc E$) in a space $S$ is a sequence of continuous path maps $\left\{h_i\right\}_{i=0}^{(n-1)[n]}$, $h_i:I\to K,  [n] =\ \mbox{mod}\ n$ with no end path map.

This paper introduces proximal path cycles considered in terms of a Tanaka good covering of an Alexandrov space~\cite{Tanaka2021TiAgoodCover}, leading to extensions of the Mitsuishi-Yamaguchi Good Coverning Theorem~\cite{MitsuisheYamaguchi209goodCover} as well as extensions of the Jordan Curve Theorem~\cite{Jordan1893coursAnalyse}.

This paper considers the homotopy of paths~\cite[\S 2.1,p.11]{Switzer2002CWcomplex} in \v{C}ech proximity spaces~\cite[\S 2.5,p 439]{Cech1966} in which nonvoid sets are spatially close provided the sets have nonempty intersection and in descriptive proximity spaces~\cite{Peters2019vortexNerves} in which nonvoid sets are descriptively close, provided the sets have the same descriptions.  A biproduct of this work is the extension of recent forms of good coverings of topological spaces~\cite{Tanaka2021TiAgoodCover}~\cite{MitsuisheYamaguchi209goodCover} as well as a fivefold extension of the Jordan curve theorem~\cite{Jordan1893coursAnalyse}.

\vspace*{0.1cm}

The main results of this paper are

{\bf Theorem} ({\em cf.} Theorem~\ref{theorem:desGoodCover}).\\
For every descriptive proximity space $M$ on a finite collection of intersecting homotopic cycles,
\begin{compactenum}[(1)]
	\item $M$ has a good cover.
	\item The nerve of $M$ and the union of the sets in $M$ have the same homotopy type.
\end{compactenum}
\vspace*{0.1cm}

Theorem~\ref{theorem:proximalJordan}).\\
Every finite collection of intersecting homotopic cycles in a proximity space $M$ satisfies the Jordan curve theorem.

\section{Preliminaries}

This section introduces notation and basic concepts underlying proximal homotopy. \\  
Let $I = [0,1]$, the unit interval.  A \emph{\bf path} in a space $X$ is a continuous map $h:I\to X$ with endpoints $h(0)=x_0$ and $h(1)=x_1$~\cite[\S 2.1,p.11]{Switzer2002CWcomplex}.  A \emph{\bf homotopy of paths} $h,h': I\to X$ with fixed end points (denoted by $h\sim h'$), is a relation between $h$ and $h'$ defined by an associated continuous map $H: I \times I\to X$, where $H(s,t) = h_t(s)$ with $H(s,0) = h(s)$ and $H(s,1) = h'(s)$.  In effect, in a homotopy of paths $h,h'$, path $h$ is continuously transformed into path $h'$.  For $h\sim h'$, paths $h,h'$ are said to be {\bf homotopic paths}. 

From the \v{C}ech proximity $\near$ in \ref{ap:Cech}, we can consider the closeness of homotopy classes in a proximity space $(X,\near)$.  

\subsection{Proximally Continuous Maps and Gluing Lemma} 
This section introduces gluing lemma for proximity spaces, defined via proximally continuous maps over a pair of \v{C}ech proximity spaces defined in terms of the proximity ${\near}$ (for the details, see ~\ref{ap:Cech}).

\begin{definition} {\rm [Proximally continuous map]~\cite[p. 5]{Smirnov1952},\cite{Efremovic1952}.}\\ 
	A map $f: (X,\delta_1) \to (Y, \delta_2)$ between two proximity spaces is proximally continuous, provided $f$ preserves proximity, {\em i.e.}, $A \ \delta_1 \ B$ implies $f(A) \ \delta_2 \  f(B)$ for $A, B \in 2^X$. 
	\quad\textcolor{blue}{\Squaresteel}
\end{definition}

\begin{remark}
	Proximally continuous maps were introduced by V.A. Efremovi\v{c}~\cite{Efremovic1952} and Yu. M. Smirnov~\cite{Smirnov1952,Smirnov1952a} in 1952 and elaborated by S.A. Naimpally and B.D. Warrack~\cite{Naimpally70} in 1970. 
	\textcolor{blue}{\Squaresteel}
\end{remark}

Lemma~\ref{thm:composition} shows that the composition of two proximally continuous maps is proximally continuous but it is also true for any types of proximally continuous maps. 

\begin{lemma}\label{thm:composition} 
	Composition of two proximally continuous maps is  proximally continuous. 
\end{lemma}
\begin{proof}
	Let $f : (X, \delta_1)\to (Y,\delta_2)$ and $g:(Y,\delta_2)\to (Z,\delta_3)$ be proximally continuous maps and $A \ \delta_1 \  B$ in $X$. Then $f(A)  \Hquad  \delta_2 \  f(B)$ since $f$ is proximally continuous and $g\circ f(A) \  \delta_3 \ g \circ f(B)$, since $g$ is proximally continuous. 
\end{proof}

\begin{figure}[!ht]
	\centering
	\includegraphics[width=75mm]{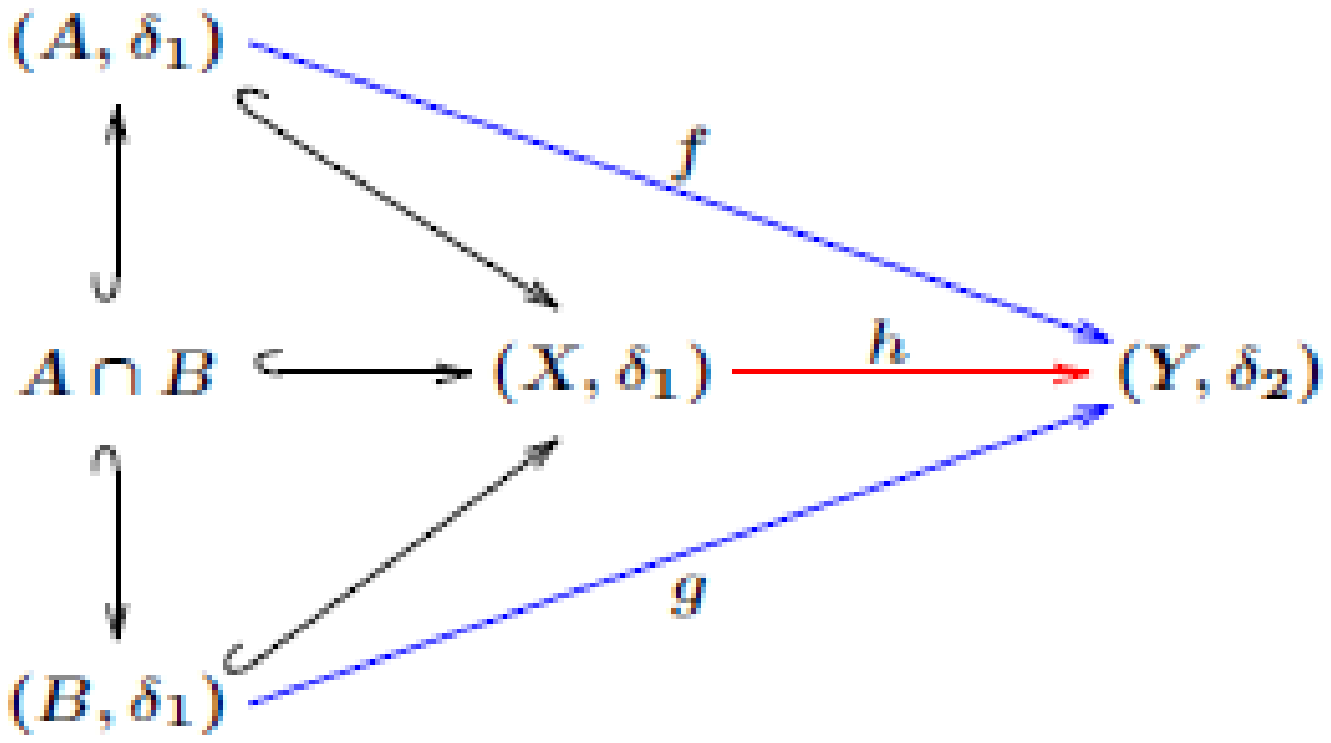}
	%
	%
	\caption[]{Gluing diagram for Proximity Spaces. Here, the black arrows represent inclusion maps and all triangles in the diagram commute.}
	\label{fig:proximalGlue}
\end{figure}

A diagram for the gluing Lemma~\ref{thm:glue} for proximity spaces is given in Fig.~\ref{fig:proximalGlue}.  This Lemma provides a basis for the proof of Theorem~\ref{thm:equivprox}.

\begin{lemma}\label{thm:glue} {\rm[Gluing Lemma for proximity spaces]} $\mbox{}$\\
	Suppose  $(X,\delta_{1})$ and $(Y,\delta_{2})$ are   proximity spaces and $A$ and $B$ are closed subsets of $X$ such that $A\cup B=X$. If $f:  (A,\delta_{1}) \to (Y, \delta_{2})$ and $g:  (B,\delta_{1}) \to (Y, \delta_{2})$ are proximally continuous maps such that $f(x)=g(x)$ for all $x\in A\cap B$, then the map $h: (X,\delta_{1}) \to (Y, \delta_{2})$ defined by 
	\[
	h(x)= \begin{cases}  f(x), & x\in A, \\  g(x), & x\in B \end{cases}	
	\]
	is also proximally continuous.
\end{lemma}

\begin{proof}
	Let $C,D$ be subsets of $X$ such that $C \ \delta_1 \ D$ so that these two sets are near.  That is, there exist $c\in C$ and $d\in D$ that are either equal  $c=d$ or near to each other $\{c\} \ \delta_1 \ \{d\}$.  If $c=d$, then we are done.\\
	
	Assume  $\{c\} \ \delta_{1} \ \{d\}$. Note that $c\in A  \ (\in B)$ implies $d \in A \ (\in B)$, since $A \ (B)$ is closed. Therefore we have the following three cases.
	\begin{description}
		\item[Case~1] $c,d \in A$.\\
		In that case, we have  $h(\{c\})=f(\{c\})  \   \delta_{2} \  h(\{d\})=f(\{d\})$ so that $h(C) \ \delta_{2} \ h(D)$.\\
		
		\item[Case~2] $c,d \in  B$.  \\
		In that case, we have   $h(\{c\})=g(\{c\})  \   \delta_{2} \  h(\{d\})=g(\{d\})$ so that $h(C) \ \delta_{2} \ h(D)$. \\
		
		\item[Case~3] $c,d \in A \cap B$. \\
		In that case,  we have   $h(\{c\})=f(\{c\})=g(\{c\})  \   \delta_{2} \  h(\{d\})=f(\{d\}))=g(\{d\})$  so that $h(C) \ \delta_{2} \ h(D)$.\\
	\end{description}
	In all cases, $h$ satisfies the proximal continuity property. 
\end{proof}

\subsection{Descriptive Proximity spaces}$\mbox{}$\\
Let $(X,\delta_{\Phi})$ be a descriptive  proximity space (see Appendix~\ref{app:dnear}). Then the \emph{descriptive closure of $A \subset X$} (denoted by $\cl_\Phi A$) is the set of all points in $X$ descriptively near to $A$, {\em i.e.},  
\begin{align*}
\cl_\Phi A &= \{ x\in X \ : \ x \ \delta_{\Phi} \ A  \}  \\
&=\{ x \in X \ : \ \Phi(x) \in \Phi(A)\}.
\end{align*}

\noindent Note that  $A$ is \emph{descriptively closed}, provided $\cl_\Phi A = A$.  \\

The following corollary is straightforward.

\begin{corollary} \label{cor:desclosed}
	Suppose $A$ is a descriptively closed subset of a descriptive proximity space  $(X,\delta_{\Phi})$. Then 
	\[
	x\in A \ \Leftrightarrow \  \Phi(x) \in \Phi(A).
	\]
	\textcolor{blue}{\Squaresteel}
\end{corollary}

\begin{definition}  {\rm [Descriptive intersection] \cite{Peters2013mcsintro} } $\mbox{}$\\
	The descriptive intersection $A\ \dcap\ B$ of two nonempty subsets $A$ and $B$ of a descriptive  proximity space $(X,\delta_{\Phi})$, is the set of all points in $A\cup B$ such that $ \Phi(A)$ and $\Phi(B)$ have common descriptions, i.e.
	\[
	A\ \dcap\ B = \left\{x\in A\cup B: \Phi(x) \in \Phi(A)\ \cap\ \Phi(B)\right\}.
	\] 
	\textcolor{blue}{\Squaresteel}
\end{definition}

\begin{definition}  {\rm [Descriptive proximally continuous maps]}\label{def:dpc} $\mbox{}$\\
	A map $f: (X, \delta_{\Phi_1}) \to (Y, \delta_{\Phi_2})$ is descriptive proximally continuous (dpc), provided $A \ \delta_{\Phi_1} \ B$ implies   $f(A) \ \delta_{\Phi_2} \ f(B)$ for $A, B  \subset X$.
	\textcolor{blue}{\Squaresteel}
\end{definition}

\begin{theorem}\label{thm:dlpcComposition}
	Composition of two dpc  maps is dpc. 
\end{theorem}
\begin{proof}
	Let $f : (X, \delta_{\Phi_1})\to (Y,\delta_{\Phi_2})$ and $g:(Y,\delta_{\Phi_2})\to (Z,\delta_{\Phi_3})$ be dpc maps and $A \ \delta_{\Phi_1} \  B$ in $X$. Then $f(A)  \Hquad  \delta_{\Phi_2}  \  f(B)$, since $f$ dpc and  $g\circ f(A) \  \delta_{\Phi_3}  \ g \circ f(B)$ since $g$ is dpc. 
\end{proof}

\begin{figure}[!ht]
	\centering
	\includegraphics[width=75mm]{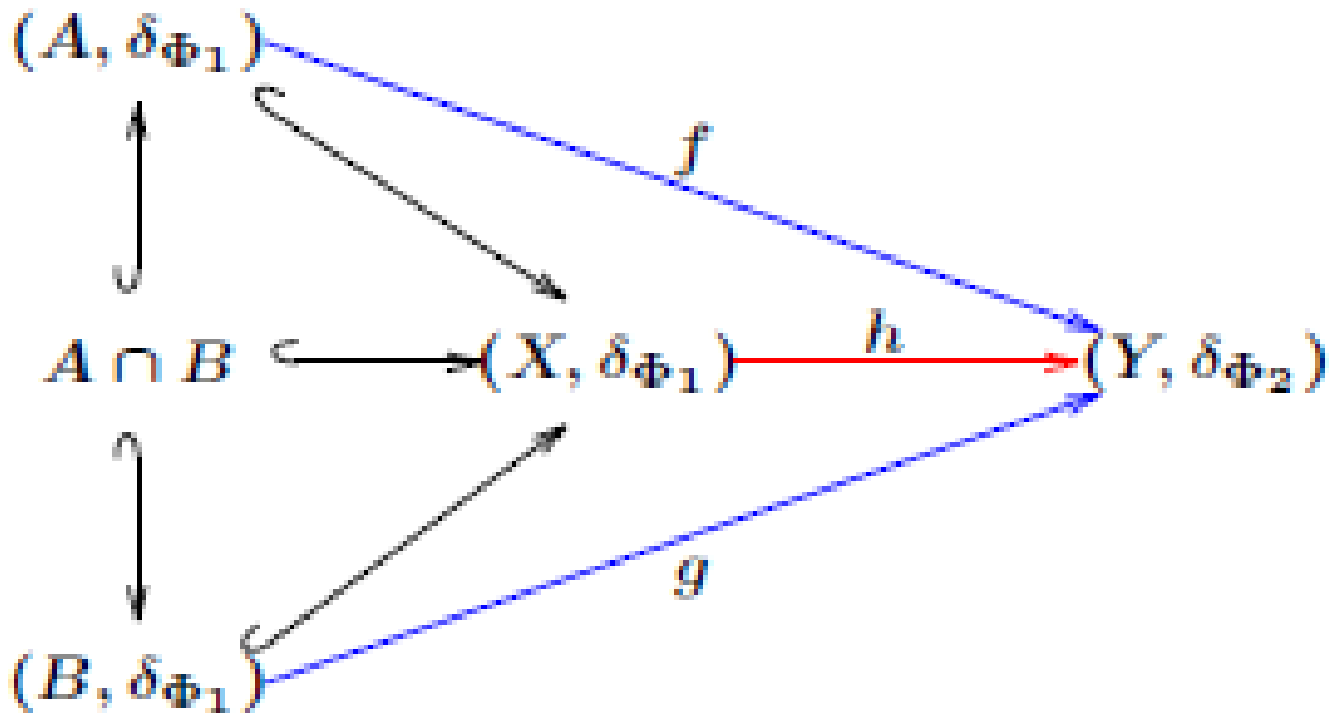}
	%
	%
	\caption[]{Gluing diagram for Descriptive Proximity Spaces. Here, the black arrows represent inclusion maps and all triangles in the diagram commute.}
	\label{fig:LproximalGlue}
\end{figure}

We adapt the gluing Lemma~\ref{thm:glue} for descriptive  proximally continuous maps.

\begin{theorem} {\rm [Descriptive Gluing]} \label{thm:desglue}\\
	Let $(X, \delta_{\Phi_1})$ and $(Y, \delta_{\Phi_2})$ be two  descriptive proximity spaces and let $A$ and $B$ be two descriptively closed subsets of $X$ with $A \cup B = X$. If $f: (A, \delta_{\Phi_1}) \to (Y, \delta_{\Phi_2})$ and $g: (B, \delta_{\Phi_1}) \to (Y, \delta_{\Phi_2})$ are dpc maps such that $f(x)=g(x)$ for all $x \in A\cap B$, then the map $h: (X, \delta_{\Phi_1})\to (Y, \delta_{\Phi_2})$ is defined by
	\[
	h(x)= \begin{cases}  f(x), & \Phi_1(x)\in \Phi_1(A) \quad   ( \equiv x \in A  \ \mbox{by}  \  Corollary~\ref{cor:desclosed}),
	\\  g(x), &  \Phi_1(x)\in \Phi_1(B)  \quad   ( \equiv x \in B   \ \mbox{by} \  Corollary~\ref{cor:desclosed}) \end{cases}	
	\]
	is also dpc.
\end{theorem}
\begin{proof}
	Let $C,D$ be subsets of $X$ such that $C \ \delta_{\Phi_1} \ D$  (so,  these two sets are descriptively near).  That is, there exist $c\in C$ and $d \in D$ that are either equal  $c=d$ or descriptively near to each other $\{c\} \ \delta_{\Phi_1} \ \{d\}$.  If $c=d$, then we are done.
	
	\noindent Assume $\{c\} \ \delta_{\Phi_1} \ \{d\}$.  Note that $c\in A  \ (\in B)$ implies $d \in A \ (\in B)$ since $A \ (B)$ is descriptively closed. Therefore we have the following three cases.
	\begin{description}
		\item[Case~1] $c,d \in A$.\\
		In that case, we have  $h(\{c\})=f(\{c\})  \   \delta_{\Phi_2} \  h(\{d\})=f(\{d\})$  so that  $h(C) \  \delta_{\Phi_2} \ h(D)$.\\
		
		\item[Case~2] $c,d \in  B$.  \\
		In that case, we have   $h(\{c\})=g(\{c\})  \    \delta_{\Phi_2}  \  h(\{d\})=g(\{d\})$  so that $h(C) \  \delta_{\Phi_2}  \ h(D)$. \\
		
		\item[Case~3] $c,d \in A \cap B$. \\
		In that case, we have   $h(\{c\})=f(\{c\})=g(\{c\})  \    \delta_{\Phi_2}  \  h(\{d\})=f(\{d\}))=g(\{d\})$  so that $h(C) \  \delta_{\Phi_2}  \ h(D)$.\\
	\end{description}
	In all cases, $h$ satisfies the descriptive proximal continuity property. 
\end{proof}

\section{Proximal Homotopy}
For two proximity spaces $(X,\delta_1)$ and $(Y,\delta_2)$,  let  $X\times Y$ denote their product.  Then the subsets   $A \times B$ and  $C \times D$ of $X\times Y$ are near, provided  $A \ \delta_1 \ C$ and $B\ \delta_2 \ D$.

\begin{definition} {\rm  [Proximal Homotopy]} $\mbox{}$\\
	Let $(X,\delta_1)$ and $(Y,\delta_2)$  be proximity  spaces and $f,g: (X,\delta_1) \to (Y,\delta_2)$ proximally continuous maps. Then we say $f$ and $g$ are proximally homotopic,  provided there exists a proximally continuous  map $H: X \times [0,1] \to Y$ such that  $H(x,0)=f(x)$  and  $H(x,1)=g(x)$. Such a map $H$ is called an proximal homotopy between $f$ and $g$.  In keeping with Hilton's notation~\cite{Hilton1952homotopy}, we write  $f \mathop{\sim}\limits_{\delta} g$, provided there is  a proximal homotopy between them.
	\textcolor{blue}{\Squaresteel}
\end{definition}

\begin{proposition} \label{thm:equivprox}
	Every proximal homotopy relation is an equivalence relation.
\end{proposition}
\begin{proof}
	A check that $\mathop{\sim}\limits_{\delta}$ is reflexive and symmetric is straightforward. 
	
	\noindent Now let $F$ and $G$ be  proximal homotopies between $f$ and $g$ and  between $g$ and $h$, respectively. Then the function $H: X \times [0,1] \to Y$ defined by 
	\[
	H(x,t)= \begin{cases}  F(x, 2t), & t \in [0,\frac{1}{2}]     \\   G(x, 2t-1), & t\in [\frac{1}{2},1]   \end{cases}
	\]
	is proximally continuous by Theorem~\ref{thm:glue},  so that this defines an proximal homotopy between $f$ and $h$. 
\end{proof}

\begin{definition} {\rm [Relative proximal Homotopy]} $\mbox{}$\\ 
	Let $(X,\delta_1)$ and $(Y,\delta_2)$  be proximity spaces and $A \subset X$. Then two proximally continuous  maps $f, g: (X,\delta_1) \to (Y,\delta_2)$  are said to be proximally homotopic relative to $A$, provided there exists an proximal homotopy $H$ between $f$ and $g$ such that $H(a,t)=f(a)=g(a)$ for all $a \in A$ and $t\in [0,1]$.  We write $f \mathop{\sim}\limits_{\delta} g \ \mbox{(rel A)}$, provided there is a proximal homotopy relative to $A$. 
	\textcolor{blue}{\Squaresteel}
\end{definition}

\begin{proposition}
	Suppose $f,g: (X, \delta_1) \to (Y, \delta_2)$ are proximally homotopic.  If  $h: (Y, \delta_2) \to (Z, \delta_3)$ is proximally continuous, then the maps $h\circ f$ and $h\circ g$ are also proximally homotopic.
\end{proposition}
\begin{proof}
	Let $F: X \times [0,1] \to Y$ be the proximal homotopy between $f$ and $g$ so that $F(x,0)=f(x)$ and $F(x,1)=g(x)$. Note that $h\circ f$ and $h\circ g$ are proximally continuous by Lemma~\ref{thm:composition} and  the map $H: X \times [0,1] \to Y$ defined by $H(x,t)=h\circ F(x,t)$ is the desired proximal homotopy between them. 
\end{proof}

\begin{proposition}
	Suppose $f,g: (X, \delta_1) \to (Y, \delta_2)$ are proximally homotopic.  If  $k: (W, \delta_0) \to (X, \delta_1)$ is proximally continuous,  then the maps $f\circ k$ and $g\circ k$ are also proximally homotopic.
\end{proposition}
\begin{proof}
	Let $F: X \times [0,1] \to Y$ be the proximal homotopy between $f$ and $g$ so that $F(x,0)=f(x)$ and $F(x,1)=g(x)$. Note that $f\circ k$ and $g\circ k$ are proximally continuous by Lemma~\ref{thm:composition} and  the map $K: Z \times [0,1] \to Y$ defined by $K(z,t)=F(k(z),t)$ is the desired proximal homotopy between them. 
\end{proof}

\begin{definition}
	A proximally continuous map is proximally nullhomotopic, provided it is proximally homotopic to a constant map. 
	\textcolor{blue}{\Squaresteel}
\end{definition}

\begin{definition}
	A proximity space is proximally contractible, provided the identity map on it is proximally homotopic to a constant map. 
	\quad\textcolor{blue}{\Squaresteel}
\end{definition}

\begin{definition}
	Two proximity spaces $(X, \delta_1)$ and $(Y, \delta_2)$ are proximally homotopy equivalent, provided there exist  proximally continuous maps $f: (X, \delta_1) \to (Y, \delta_2)$ and $g: (Y, \delta_2) \to (X, \delta_1)$ such that $g\circ f$  and $f \circ g$ are proximally homotopic to the identity maps on $X$ and $Y,$ respectively. 
\end{definition}

\subsection{Homotopy between descriptive proximally continuous maps}$\mbox{}$\\
The results for pairs of proximity spaces given so far hold for proximity spaces without restrictions.  
\begin{proposition}\label{thm:dlpProduct}
	The product of descriptive proximity spaces is a descriptive proximity  space.
\end{proposition} 
\begin{proof}
	Let $\{(X_i, \delta_{\Phi_i})\}_{i\in J}$ be a family of descriptive proximity spaces spaces, where $J$ is an index set. Then we can define a descriptive nearness relation $\delta_{\Phi}$ on the product space  $X:= \prod_{i\in J} X_i$  with the probe function $\Phi:=\prod_{i\in J} \Phi_i $ by declaring that two subsets $A, B$ of $X$ are descriptively near, provided $A \  \delta_{\Phi} \  B$ if and only if $\mathrm{pr}_i(A) \ \delta_{\Phi_i} \ \mathrm{pr}_i(B)$ for all $i \in J$, where $\mathrm{pr}_i$ is the i$^{th}$  projection map of $X$ onto $X_i$.
\end{proof}


\begin{remark}\label{rem:descriptiveNearnessRelation}
	To define the descriptive homotopy between dpc maps, we impose a descriptive nearness relation on the closed interval $[0,1]$ in the following manner. Two subsets $A$ and $B$ of $[0,1]$ are descriptively near, provided $D(A,B)=0$ (that is, the descriptive proximity relation and the (metric) proximity relation coincide).
	\textcolor{blue}{\Squaresteel}
\end{remark}

The descriptive nearness relation introduced in Remark~\ref{rem:descriptiveNearnessRelation}
leads to descriptive homotopic maps.\\

\begin{definition} {\rm  [{\bf Descriptive Proximal Homotopy}]} $\mbox{}$ \\
	Let $(X,\delta_{\Phi_1})$ and $(Y,\delta_{\Phi_2})$  be descriptive proximity spaces and $f, g: (X,\delta_{\Phi_1}) \to (Y,\delta_{\Phi_2})$ dpc  maps. Then we say $f$ and $g$ are  descriptive proximally homotopic, provided there exists a dpc  map $H: X \times [0,1] \to Y$ such that $H(x,0)=f(x)$  and  $H(x,1)=g(x)$. Such a map $H$ is called a descriptive proximal homotopy between $f$ and $g$. We denote $f \mathop{\sim}\limits_{\Phi}  g$, provided  there exists a descriptive proximal homotopy between them.
	\textcolor{blue}{\Squaresteel}
\end{definition}


\begin{proposition}\label{thm:LproximalHomotopy}
	Every descriptive proximal homotopy relation is an equivalence relation.
\end{proposition}
\begin{proof}
	It's easy to check that $ \mathop{\sim}\limits_{\Phi}$ is reflexive and symmetric. 
	
	\noindent Let $F$ and $G$ are the descriptive proximal homotopies between $f$ and $g$ and  between $g$ and $h$, respectively. Then the function $H: X \times [0,1] \to Y$ defined by 
	\[
	H(x)= \begin{cases}  F(x, 2t), & t \in [0,\frac{1}{2}],     \\   G(x, 2t-1), & t\in [\frac{1}{2},1]   \end{cases}
	\]
	is dpc by Theorem~\ref{thm:desglue},  so that this defines a descriptive proximal homotopy between $f$ and $h$.  
\end{proof}

\begin{definition} {\rm [{\bf Descriptive proximal Relative Homotopy}]} $\mbox{}$ \\
	Let $(X,\delta_{\Phi_1})$ and $(Y,\delta_{\Phi_2})$  be descriptive proximity spaces and $A \subset X$. Then  two dpc maps $f, g: (X,\delta_{\Phi_1}) \to (Y,\delta_{\Phi_2})$  are said to be descriptive proximally homotopic relative to $A$, provided there exists a descriptive proximal homotopy $H$ between $f$ and $g$ such that $H(a,t)=f(a)=g(a)$ for all $a \in A$ and $t\in [0,1]$.  We write $f \mathop{\sim}\limits_{\Phi}   g \ \mbox{(rel A)}$, provided there is a descriptive proximal homotopy relative to $A$. \textcolor{blue}{\Squaresteel}
\end{definition}


\subsection{Paths in proximity spaces}
\label{subsec:pathsanddescriptivepaths}

From Remark~\ref{rem:descriptiveNearnessRelation}, we know that the descriptive nearness relation also induces a descriptive proximity relation on $[0,1]$. This leads to the introduction of (descriptive) proximal paths in a (descriptive) proximity space. 


\begin{definition} {\rm [{\bf Proximal Path}]}$\mbox{}$\\
	Let  $(X,\delta)$ be a proximity space and $x_0, x_1 \in X$.  Then a {\bf proximal path} between $x_0$ and $x_1$ is an proximally continuous map $\alpha : [0,1] \to X$  such that $\alpha(0)=x_0$ and $\alpha(1)=x_1$, {\em i.e.}, for two subsets of $A, B$ in $[0,1]$, $D(A,B)=0$ implies $\alpha(A)\  \delta \ \alpha(B)$.
	\textcolor{blue}{\Squaresteel}
\end{definition}

\noindent In this section, we introduce constant proximal paths and their descriptive forms.

\begin{definition} {\rm[{\bf Constant Proximal Path}]}$\mbox{}$\\
	For a proximity space $(X,\delta)$, the constant proximal path $c: [0,1] \to X$  at  $x_0\in X$  is the proximal path such that $c(t)=x_0$ for every $t\in [0,1]$. 
	\textcolor{blue}{\Squaresteel}
\end{definition}

\begin{definition}\label{def:descriptiveProxPath}
	{\rm	[{\bf Descriptive proximal path}]} $\mbox{}$\\
	Let  $(X,\delta_\Phi)$ be a descriptive proximity space and $x_0, x_1 \in X$.  Then a \emph{descriptive proximal path} between $x_0$ and $x_1$ is a dpc map $\alpha : [0,1] \to X$ such that $\alpha(0)=x_0$ and $\alpha(1)=x_1$, {\em i.e.}, for two subsets of $A, B$ in $[0,1]$, $D(A,B)=0$ implies $\alpha(A)\  \delta_\Phi\ \alpha(B)$.
	\textcolor{blue}{\Squaresteel}
\end{definition}

Descriptive  proximally continuous maps were informally introduced in~\cite{Peters2016ComputationalProximity},
defined here in terms of path descriptions, utilizing the descriptive proximity relation $\dnear$ (see ~\ref{app:dnear}).\\

\begin{definition} {\rm [{\bf Path description}]}$\mbox{}$ \\
	Let $h,k$ be proximally homotopic paths in a proximity space $X$.
	\begin{align*}
	\boldsymbol{\Phi(h)} &= \overbrace{\left\{\ \Phi(h(s)):\ s \in [0,1] \right\} \subseteq \mathbb{R}^n}^{\mbox{\textcolor{blue}{\bf set of  feature vectors that  describe path $h$}}}.\\
	\boldsymbol{\Phi(h) = \Phi(k)} &\overbrace{\Leftrightarrow  \boldsymbol{h\ \dnear\ k}.}^{\mbox{\textcolor{blue}{\bf descriptively close paths}}}\\
	\end{align*}
	
	Similarly, for descriptively close homotopy classes $[h], [k]$, we write 
	\[
	\Phi([h]) = \Phi([k]) \overbrace{\Leftrightarrow  \boldsymbol{[h]\ \dnear\ [k].}}^{\mbox{\textcolor{blue}{\bf descriptively close path classes}}}\mbox{\quad\textcolor{blue}{\Squaresteel}}
	\]
\end{definition}

In other words, the closeness of descriptions of paths (and path classes) is expressed using the descriptive proximity relation $\dnear$.  \\

\begin{definition}{\rm [{\bf Descriptive Proximally Continuous}]}$\mbox{}$ \\
	Let $[h],[k]$ be nonempty classes of paths in a  proximity space $X$.  A map $f:(2^X\times I,\dnear)\to (2^X\times I,\dnear)$ is \emph{descriptive proximally continuous} (dpc), provided
	\[
	[h]\ \dnear\ [k]\ \mbox{implies}\ f([h])\ \dnear\ f([k]).
	\]
\end{definition}

Unlike the constant proximal path, descriptive proximal paths (from Def.\ref{def:descriptiveProxPath}) fall into two niches, namely, {\bf (ordinary descriptive) constant paths} and {\bf degenerate  descriptive constant paths}, introduced in this section. 
These proximal paths lead to introduction of descriptive contractibility and an extended form of Tanaka good cover.\\


\begin{figure}[!ht]
	\centering
	\includegraphics[width=75mm]{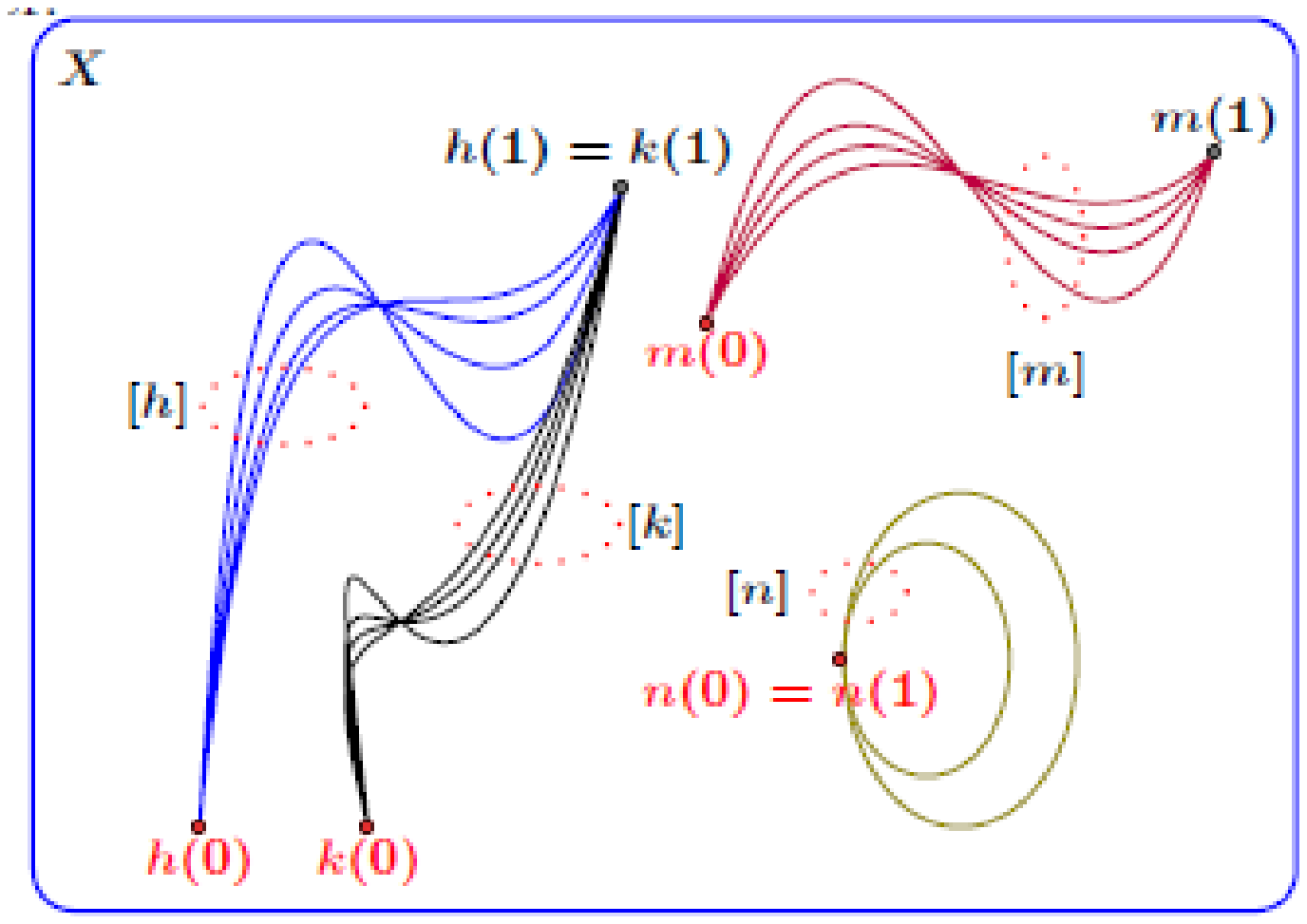}
	\caption[]{The identity map on $H(X)$ is degenerate descriptive constant.}
	\label{fig:example}
\end{figure}

\begin{definition} {\rm [{\bf Descriptive constant map}]} \\
	Let  $(X,\delta_{\Phi_1})$ and $(Y,\delta_{\Phi_2})$ be  descriptive proximity spaces. Then a map $d: X \to Y$ is said to be a  descriptive constant, provided,   $d(x)=y_0$ for all $x \in X$ and for some $y_0 \in Y$. 	$\mbox{\textcolor{blue}{\Squaresteel}}$
\end{definition}

\begin{definition} \label{def:desContractible}
	{\rm [{\bf Descriptively contractible space}]} $\mbox{}$ \\
	A descriptive proximity space is descriptive proximally contractible, or descriptively contractible for short,   provided, the identity map on it is descriptive proximally homotopic to a descriptive constant map.
	$\mbox{\textcolor{blue}{\Squaresteel}}$
\end{definition}


\begin{definition}\label{def:degnerateDescrProxConstMap}
	{\rm [{\bf Degenerate Descriptive Constant Map}]} $\mbox{}$\\
	Let  $(X,\delta_{\Phi_1})$ and $(Y,\delta_{\Phi_2})$ be  descriptive proximity spaces. Then a map $d: X \to Y$ is said to be a  degenerate descriptive constant, provided  $\Phi_2(d(x_0))=\Phi_2(d(x_1))$ for all $x_0,x_1 \in X$. 
\end{definition}

From Def.~\ref{def:degnerateDescrProxConstMap}, observe that the degenerate descriptive constant map need not map every element to a fixed element, but instead it fixes the description. That is $\abs{\mathrm{im} \ d} \geq 1$ but $\abs{\Phi_2(d(X))}=1$ so that the image of $\Phi_2\circ d$ consists of a single element, say $\ast \in \mathbb{R}^n$  (see Figure~\ref{fig:degconst}). We say that $d$ is an ordinary descriptive constant map, provided   $\abs{\mathrm{im} \ d} = 1$.

\begin{figure}[!ht]
	\centering
	\begin{tikzcd}
		\centering
		X  \arrow[r, "d"] & Y \arrow[r, "\Phi_2"] & \{\ast\} \subset \mathbb{R}^n 
	\end{tikzcd}
	\caption{ $\Phi_2\circ d$ is a constant map on $X$, provided $d$ is  degenerate descriptive constant.}
	\label{fig:degconst}
\end{figure}

\begin{example}
	Let $H(X)$ denote the path homotopy classes in $X$ given in  Figure~\ref{fig:example} and  the paths in each of the homotopy class be described in terms of the color of their initial points. Then the identity map $id: (H(X), \Phi) \to (H(X),\Phi)$ is a degenerate descriptive constant map since the initial points of all  paths are red. 
\end{example}

\begin{theorem}
	A degenerate descriptive constant map is a dpc map.
\end{theorem}
\begin{proof}
	For two subsets $A$ and $B$ of $X$, suppose that $A \ \delta_{\Phi_1} \ B$. From Def.~\ref{def:degnerateDescrProxConstMap} for a degenerate descriptive constant map, we have  $\Phi_2(d(A))=\Phi_2(d(B))$ so that $d(A) \ \delta_{\Phi_2} \ d(B)$, which completes the proof.
\end{proof}

\begin{definition}\label{def:degdesContractible}	{\rm [{\bf Degenerate Descriptively Contractible Space}]} $\mbox{}$\\
	A descriptive proximity space is a degenerate descriptively  contractible,   provided, the identity map on it is descriptive proximally homotopic to a degenerate descriptive constant map.
\end{definition}

%
\begin{proposition} \label{prop:descont}
	Suppose that $(X,\dnear)$ is a descriptive proximity space and  $c_d$  is  a degenerate descriptive constant map  on $X$ with $x_0 \in \text{Im}(c_d)$.  Then $c_d$ and the descriptive constant map  $c_{x_0}$  at $x_0$  are descriptively homotopic. 
\end{proposition}
\begin{proof}
	The desired homotopy  $H: X \times I \to X$ is a map such that $H(x,t)=x_0$. 
\end{proof}

Since descriptive proximal relation is transitive, we have the following corollary.

\begin{corollary} \label{cor:degenerate}
	A degenerate descriptively contractible space is also a descriptively contractible.
\end{corollary}

\begin{figure}[!ht]
	\centering
	\includegraphics[width=75mm]{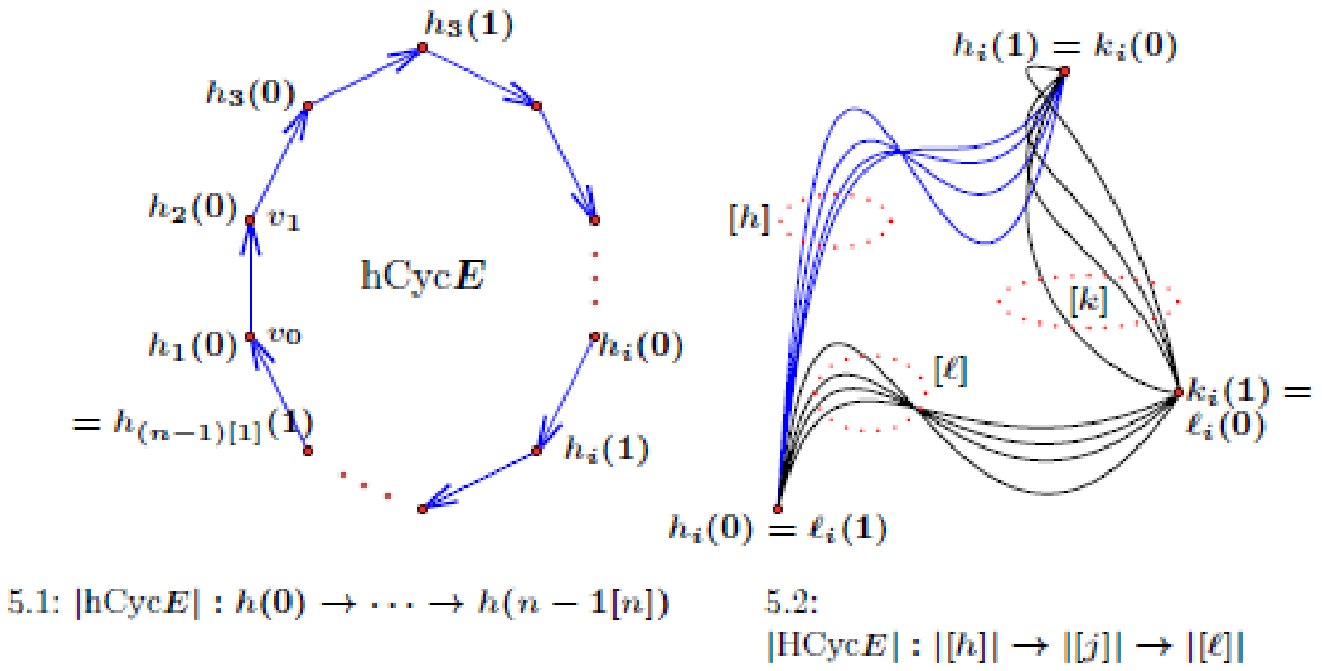}
	\caption[]{Two Forms of Homotopic cycles}
	\label{fig:homotopicCycles}
\end{figure}

\begin{figure}[!ht]
	\centering
	\includegraphics[width=75mm]{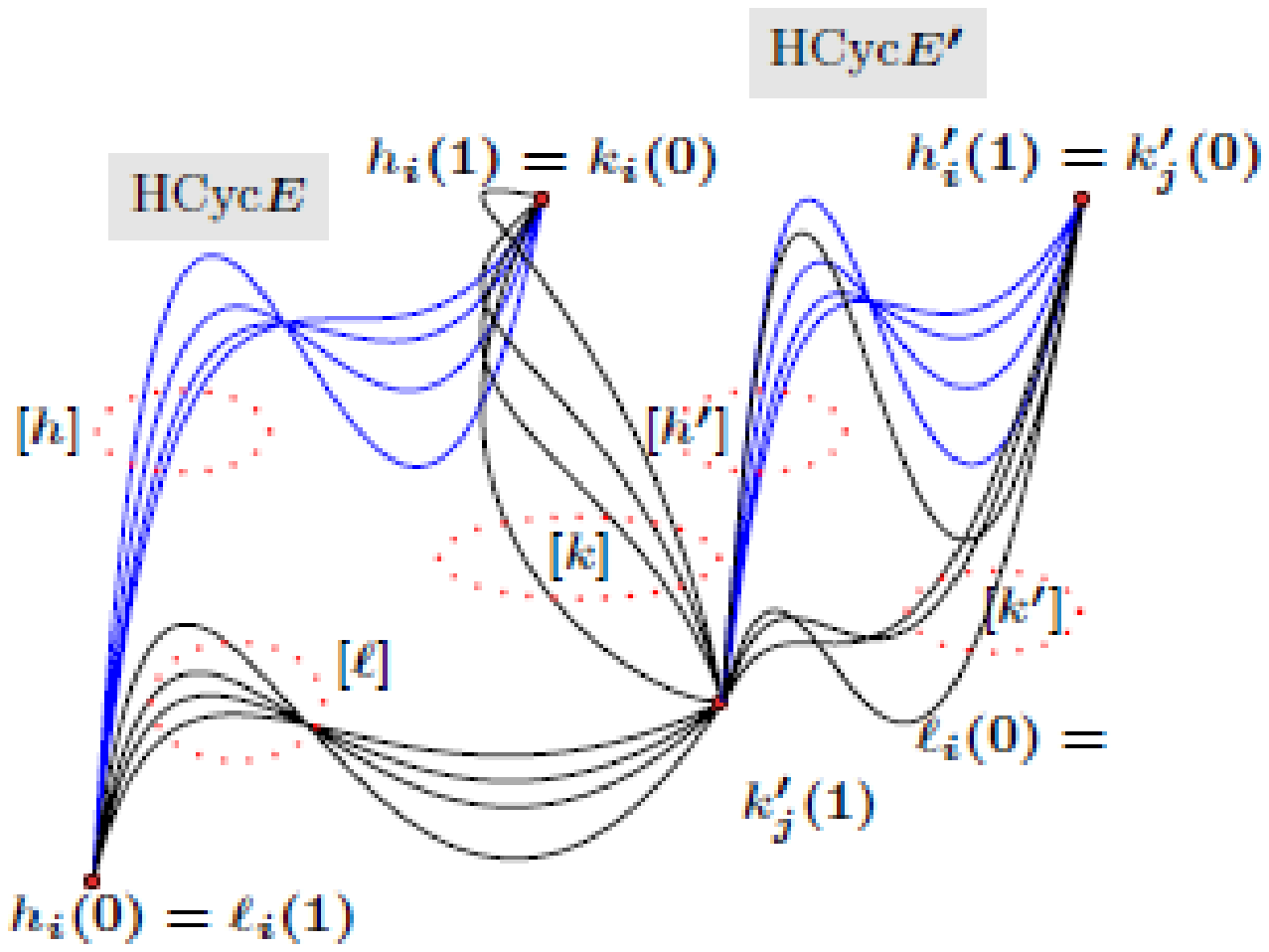}
	\caption[]{$\boldsymbol{\abs{\hSys E}}$, a homotopic cycle system.}
	\label{fig:HcNerve}
\end{figure}

\section{Homotopic Cycles}
This section introduces three forms of homotopic cycles (briefly, {\bf path cycles}, namely, simple path cycles, multi-path cycles and path cycle systems. Geometrically, a path cycle has the appearance of the boundary of a Vigolo Hawaiian earring~\cite{Vigolo2018HawaiianEarrings}. These path cycles lead to extensions of the Jordan Curve Theorem.

Recall that a path in a space $X$ is a continuous map $h:I\to X$~\cite[\S 2.1,p.11]{Switzer2002CWcomplex}.

\begin{definition}\label{def:filledCycle} {\rm [{\bf Simple Path Cycle}] } $\mbox{}$\\
	In a space $X$ in the Euclidean plane, let $h:I\to X$ be a path (briefly, hpath). A path cycle $E$ (denoted by $\hcyc E$) is a collection of hpath-connected vertexes attached to each other with no end vertex.
	\qquad\textcolor{blue}{\Squaresteel}
\end{definition}

\begin{remark}
In general, path cycles have empty interiors.
However, in the case where a path cycle is constructed on the boundary of a physical shape recorded by a camera, the path cycle will have a nonvoid interior.  From a descriptive proximity perspective, the nonvoid interior is useful in cases where one or more features in a feature vector quantify distinuishing characteristics of parts of a shape interior such as centroidal vector field.  See an illustration, see Example~\ref{ex:featureVector} in Appendix~\ref{app:dnear}.
\qquad\textcolor{blue}{\Squaresteel}
\end{remark}

\begin{example}
	A geometric realization of a simple path cycle $\abs{\hcyc E}$ is shown in Fig.~\ref{fig:HcNerve}.  Each edge in $\abs{\hcyc E}$ is an hpath $\abs{h_i}, i\in [0,\dots,n-1[n]]$. 
	\qquad\textcolor{blue}{\Squaresteel}
\end{example}

An enriched form of a path cycle is derived from the paths in homotopic classes that provide path-connected cycle vertexes.

\begin{definition}\label{def:HCycle} {\rm  [Multi-Path Homotopic Cycle]} $\mbox{}$\\
	In a space $X$ in the Euclidean plane, let $[h]$ be a homotopic class containing multiple hpaths. A multi-path homotopic cycle $E$ (denoted by $\Hcyc E$) is a collection of homotopic classes containing hpaths-   connected vertexes attached to each other with no end vertex and $\Hcyc E$ has a nonvoid interior.
	\qquad\textcolor{blue}{\Squaresteel}
\end{definition}

\begin{example}
	A geometric realization of a multi-path homotopic cycle $\abs{\Hcyc E}$ is shown in Fig.~\ref{fig:HcNerve}.  There are multiple homotopic paths between each pair of vertexes in $\abs{\Hcyc E}$. For example, between vertexes $\abs{h_i(0)},\abs{h_i(1)}$, there are multiple h-paths in class $[h]$.
	\qquad\textcolor{blue}{\Squaresteel}
\end{example}

A system of path cycles results from a collection of $\Hcyc$-cycles that have nonvoid intersection.

For a space $X$ in the Euclidean plane, let $H(X)$ denote the set of all homotopy classes $[h]$ in $X$.

\begin{definition}\label{def:filledCycle2} {\rm  [Homotopic Cycle System]} $\mbox{}$\\
	In a space $X$ in the Euclidean plane, a homotopic cycle system $E$ (denoted by $\hSys E$) is a collection of $\Hcyc$-cycles such that
	\[
	\hSys E = \left\{\Hcyc E\in 2^{H(X)}    : \bigcap \Hcyc E = \mbox{vertex}\ v\in  H(X)  \right\}.\ \mbox{\qquad\textcolor{blue}{\Squaresteel}}
	\]
\end{definition}

\begin{example}
	A geometric realization of a path cycle system $\abs{\hSys E}$ is shown in Fig.~\ref{fig:HcNerve}.  This system contains a pair of multi-path cycles $\Hcyc E$, $\Hcyc E'$ attached to each other, {\em i.e.}, we have
	\begin{align*}
	\ell_i &\in [\ell]\in \Hcyc E,\\
	k'_j &\in [k]\in \Hcyc E',\\
	\hSys E &= \{ \Hcyc E, \Hcyc E' \}\\
	\Hcyc E\cap\Hcyc E' &= \ell_i(0) = k'_j(1).
	\mbox{\qquad\textcolor{blue}{\Squaresteel}} 
	\end{align*}
\end{example}

\section{Good coverings and Jordan Curve Theorem extension}
This section introduces good coverings of descriptive proximity spaces and an extension of the Jordan Curve Theorem in terms of the boundary of a path cycle.

\begin{definition}\label{def:CycleClosure} {\rm [Closure in a Hausdorff metric space} ]. $\mbox{}$\\
	Let $A\in 2^X$ (nonvoid subset $A$ in a Hausdorff metric space~\cite{Hausdorff1914,Hausdorff1914b} $X$) and $D(x,A) = \inf\left\{\abs{x-a} :   a\in A  \right\}$ be the Hausdorff distance between a point $x\in X$ and subset $A$~\cite[\S 22,p. 128]{Hausdorff1914b}.  The closure of $A$~\cite[\S 1.18,p. 40]{DBLP:series/isrl/2014-63} is defined by
	\[
	\cl(A) = \left\{x\in X: D(x,A) = 0\right\}.\ \mbox{\qquad\textcolor{blue}{\Squaresteel}}
	\]
\end{definition}

\begin{definition}\label{def:closure} For a Hausdorff metric  space $X, A\in 2^X$, let $\cl A$ be the closure of $A$.  Then the boundary of $A$ (denoted by $\bdy A$) is the set of all points on the border of $\cl A$ and not in the complement of $\cl A$ (denoted by $\partial \cl A$).  Also, the interior of $A$ (denoted by $\Int A$) is the set of all points in $\cl A$ and not on the boundary of $A$, {\em i.e},
	\begin{align*}
	\partial(\cl A) &= X\setminus \cl A,\ \mbox{all points in $X$ and not in $\cl A$}.\\
	\Int(A) &= \left\{E\in 2^X: E\subset \cl A\ \mbox{and}\ E \cap \bdy A = \emptyset  \right\}.\\
	\bdy(A) &= X\setminus (\Int A\cup \partial \cl A).\\ 
	\cl A &= \bdy(A)\cup\Int(A).\ \mbox{\qquad\textcolor{blue}{\Squaresteel}}
	\end{align*}
\end{definition} 

\begin{remark}
	Geometrically, a path cycle system is a necklace. The clasp of the necklace is the vertex in the intersection of the system cycles. This is the case in Fig.~\ref{fig:HcNerve}.
	\qquad \textcolor{blue}{\Squaresteel}
\end{remark}

Recall that a {\bf cover} of a space $X$ is a collection of subsets $E\in 2^X$ such that $X=\bigcup E$~\cite[\S 15.9,p. 104
]{Willard1970}. 

\begin{definition}\label{def:goodCover}{\rm {\bf Good Cover}~\cite[\S 4,p. 12]{Tanaka2021TiAgoodCover}}.\\
	A cover of a space $X$ is a {\bf good cover}, provided, $X$ 
	has a collection of subsets $E\in 2^X$ such that $X=\bigcup E$ and 
	$\mathop{\bigcap}\limits_{\mbox{finite}} 
	E\neq \emptyset$
	is contractible,
	{\em i.e.}, all nonvoid intersections of the finitely many subsets $E\in 2^X$  are  contractible.
	\qquad \textcolor{blue}{\Squaresteel}
\end{definition}

\begin{example}
	For a space $X$ in  the Euclidean plane, let $\hSys E = \left\{\Hcyc E,\Hcyc E'\right\}$ a system of path cycles in $X$ with nonempty intersection such that  a  geometric realization $\abs{\hSys E}$  is shown in Fig.~\ref{fig:HcNerve}.
	This is an example of planar Tanaka good cover of a $H(X)$, since  
	\begin{align*}
	H(X) &= \Hcyc E\cup\Hcyc E',\ \mbox{and}\\
	\Hcyc E\cap\Hcyc E' & = \ell_i(0).\ \mbox{\qquad \textcolor{blue}{\Squaresteel}} 
	\end{align*}
\end{example}



\begin{definition}\label{def:desgoodCover} 
	{\rm [Descriptively good cover]}$\mbox{}$ \\
	Let $X$ be a descriptive proximity space  with a probe function $\Phi: 2^X \to \mathbb{R}^n$. A descriptively good cover of $(X,\Phi)$ is a collection of subsets $E \in 2^X$  such that $X = \bigcup E$ and $\displaystyle \mathop{\bigcap}\limits_{ \Phi, \mbox{finite}}  E \neq \emptyset$, i.e., all nonvoid descriptive intersections of the finitely many subsets $E \in 2^X$ are descriptively contractible. \qquad \textcolor{blue}{\Squaresteel}
\end{definition}

\begin{definition}\label{def:degdesgoodCover} 
	{\rm [Degenerate Descriptively good cover]}\\
	Let $X$ be a descriptive proximity space  with a probe function $\Phi: 2^X \to \mathbb{R}^n$. A degenerate descriptively good cover of $(X,\Phi)$ is a collection of subsets $E \in 2^X$  such that $X = \bigcup E$ and $\displaystyle \mathop{\bigcap}\limits_{ \Phi, \mbox{finite}}  E \neq \emptyset$ is degenerate descriptively contractible, i.e., all nonvoid descriptive intersections of the finitely many subsets $E \in 2^X$ are  degenerate descriptively contractible. \qquad \textcolor{blue}{\Squaresteel}
\end{definition}



\begin{proposition}
	For a space $X$ in the Euclidean plane, $\hSys E$ is a good cover of $H(X)$.
\end{proposition}
\begin{proof}
	Observe that each element $\Hcyc E$ in  $\hSys E$ is a subset of $H(X)$ and   by the nature of $\hSys E$,  $H(X)=\bigcup \Hcyc E$ and $\bigcap \Hcyc E$ is a single vertex so that it is contractible. 
\end{proof}

\begin{theorem}\label{theorem:EHnerve}{\rm~\cite[\S III.2,p. 59]{Edelsbrunner1999}}.\\
	Let $F$ be a finite collection of closed, convex sets in Euclidean space.  Then the nerve of $F$ and union of the sets in $F$ have the same homotopy type.
\end{theorem}

Let $\mathop{\angle}\limits^{~}_{\kappa}bac$ denote the inner angle of a geodesic triangle of length $\abs{ab},\abs{bc},\abs{ca}$, at the vertex with opposite side of length $\abs{bc}$, in a simply connected complete surface of curvature $\kappa$.

A geodesic complete metric space M is an {\bf Alexandrov space} (of curvature bounded locally from below) 
~\cite[\S 2.1,p. 3]{MitsuisheYamaguchi209goodCover}, provided, for each $p\in M$, there exist an $r>0$ and $\kappa\in\mathbb{R}$ such that for any distinct four points $a_i\in B(p,r), i = 1,2,3,4$ with $\mbox{max}_{1\leq i<j\leq 3}\left\{\abs{a_0a_i}+\abs{a_0a_j}+\abs{a_ia_j}\right\}<\frac{\pi}{\sqrt{\kappa}}$, if $\kappa > 0$,
we have
\[
\mathop{\sum}\limits_{1\leq i\leq j\leq 3} \mathop{\angle}\limits^{~}_{\kappa}a_ia_oa_j\leq 2\pi.
\]


\begin{proposition}\label{prop:AlexandrovSpace}
	For a closed subset $X$  in the Euclidean plane with a probe function $\Phi$,
	the descriptive proximity space $\left(X,\dnear\right)$  is an Alexandrov space.
\end{proposition}
\begin{proof}
	$X$ is complete since it is a closed subset of the Euclidean plane. For an element $p \in X$, consider the unit ball $B(p,1)$ and take the points $a_1, a_2, a_3$ on the boundary of $B(p,1)$ and let $\kappa=1$ 
	(the curvature of $B(p,1)$, the reciprocal of the radius)
	Then $\abs{pa_1} + \abs{pa_2} + \abs{pa_3}=3 \leq \frac{\pi}{\sqrt{1}}$ and we have $\mathop{\angle}\limits^{~}_{\kappa}a_1pa_2 +  \mathop{\angle}\limits^{~}_{\kappa}a_1pa_3 + \mathop{\angle}\limits^{~}_{\kappa}a_2pa_3 = 2\pi$.
\end{proof}

A main result in this paper is a extension of the Mitsuishi-Yamaguchi Theorem~\ref{theorem:Mitsuishi-Yamaguchi}.

\begin{theorem}\label{theorem:Mitsuishi-Yamaguchi}{\rm ~\cite[Theorem 1.1(2),p. 8108]{MitsuisheYamaguchi209goodCover}}.\\
	Every open covering $\gamma$ of an Alexandrov space $M$ has the same homotopy type as the nerve of any good covering of $M$.
\end{theorem}


\begin{proposition}\label{prop:proxOpenCover}
	Every descriptive proximity space $\left(X,\dnear\right)$ with a probe function $\Phi: 2^X \to \mathbb{R}^n$ in the Euclidean plane has an open covering.
\end{proposition}
\begin{proof}
	For $x \in X$ and positive number $\varepsilon>0$,  define the descriptive $\varepsilon$ neighborhoud of $x$ by letting $B_\Phi(x,\varepsilon)= \{ y\in X :  d(\Phi(x), \Phi(y))<\varepsilon     \}$ 
	where $d$ is a Euclidean distance on $\mathbb{R}^n$. Observe that  $B_\Phi(x,\varepsilon)$ is open in $X$, since, for an element $y\in X$, we have $B_\Phi(y,r)\subseteq B_\Phi(x,\varepsilon)$, where $r= \varepsilon - d(\Phi(x), \Phi(y))$. Then the collection of open sets  $\{B_\Phi(x,\varepsilon) : x \in X, \varepsilon>0\}$ is an open covering of $X$. 
\end{proof}


\begin{theorem}\label{theorem:desGoodCover}
	Let $X$ be a descriptive proximity space in the Euclidean plane with an open covering and with a probe function $\Phi: 2^X \to \mathbb{R}^n$.  Also, let $H(X)$ be the collection of all homotopy classes covering space $X$ and $X = H(X)$.
	\begin{compactenum}[1$^o$]
		\item If nerve of $E\in 2^{H(X)}$ in space $X$ is descriptively contractible, then $X$ has a descriptively good cover.
		\item If nerve of $E\in 2^{H(X)}$ in space $X$ is degenerate descriptively contractible, then $X$ has a descriptively good cover.
		\item If $H(X)$ is an Alexandrov space with an open covering, then the nerve of $H(X)$ and the union of sets in $H(X)$ have the same homotopy type. 
		\item If $H(X)$ is a finite collection of closed, convex sets in Euclidean space.  Then the nerve of $H(X)$ and union of the sets in $H(X)$ have the same homotopy type.
	\end{compactenum}
\end{theorem}
\begin{proof}
	1$^o$: For $E\in 2^{H(X)}$ in $\left(X,\dnear\right)$, we have $X = \bigcup E$, since $X = H(X)$.  We also know that all nonvoid descriptive intersections of finitely many subsets $E\in 2^{H(X)}$ are descriptively contractible.  Hence, from Def.~\ref{def:desgoodCover}, $H(X)$ is a descriptively good cover of $X$.\\
	2$^o$: For $E\in 2^{H(X)}$ in $\left(X,\dnear\right)$, we have $X = \bigcup E$, since $X = H(X)$.  We also know that all nonvoid descriptive intersections of finitely many subsets $E\in 2^{H(X)}$ are degenerate descriptively contractible.  Hence, from Def.~\ref{def:degdesgoodCover}, $H(X)$ is a degenerate descriptively good cover of $X$.\\
	3$^o$: From Prop.~\ref{prop:AlexandrovSpace}, $X$ is an Alexandrov space. If $X$ has an open covering, then from Theorem~\ref{theorem:Mitsuishi-Yamaguchi}, the desired result follows.\\
	4$^o$: If $H(X)$ is a finite collection of closed, convex sets, then the desired result follows from Theorem~\ref{theorem:EHnerve}.
\end{proof}

A another main result in this paper is a fivefold extension of the Jordan curve theorem.

\begin{theorem}\label{thm:JordanCurveTheorem} {\rm [Jordan Curve Theorem~\cite{Jordan1893coursAnalyse}]}.\\
	A simple closed curve lying on the plane divides the  
	plane into two regions and forms their common boundary.
\end{theorem}


\begin{theorem}\label{theorem:proximalJordan}
	Let $\hcyc E$ (simple homotopic cycle), $\Hcyc E$ (multi-homotopic cycle), $\hSys E$ (path cycle system) be in the Euclidean plane. Then
	\begin{compactenum}[1$^o$]
		\item The boundary $\bdy(\cl(\hcyc E))$ satisfies the Jordan Curve Theorem. 
		\item The boundary $\bdy(\cl(\Hcyc E))$ satisfies the Jordan Curve Theorem.
		\item The boundary $\bdy(\cl(\hSys E))$ satisfies the Jordan Curve Theorem.
		\item If $X=H(X)$ in $\left(X,\dnear\right)$ has a descriptively good cover, then $\bdy(\cl_\Phi(H(X)))$  satisfies the Jordan Curve Theorem.
		\item If $X=H(X)$ in $\left(X,\dnear\right)$ has a degenerate descriptively good cover, then $\bdy(\cl_\Phi(H(X)))$ satisfies the Jordan Curve Theorem.
	\end{compactenum}
\end{theorem}
\begin{proof}
	1$^o$: The boundary $\bdy(\cl(\hcyc E))$ is a sequence of paths on a curve that is simple (no loops) and closed (the sequence begins and ends with the same vertex).  Hence, by Theorem~\ref{thm:JordanCurveTheorem}, $\bdy(\cl(\hcyc E))$ divides the  
	plane into two regions and forms their common boundary.\\
	
	2$^o$: Replace $\bdy(\cl(\hcyc E))$ in 1$^o$ with $\bdy(\cl(\Hcyc E))$ and the proof is symmetric with the proof of 1$^o$.\\
	
	3$^o$: Replace $\bdy(\cl(\hcyc E))$ in 1$^o$ with $\bdy(\cl(\hSys E))$ and observe that curve on each boundary $\Hcyc E\in\hSys E$ is a simple, closed curve attached to the other homotopic cycle boundaries by a single vertex.  Then curve on the boundary continues along the curves of the other cycle boundaries, forming an elongated curve that is both simple and closed. Hence, the boundary $\bdy(\cl(\hSys E))$ satisfies Theorem~\ref{thm:JordanCurveTheorem}\\
	
	4$^o$: Observe that if $E \in 2^{H(X)}$  is an element in a descriptively good covering of $H(X)$,   then it is descriptively contractible. This is equivalent to saying that $E$ contains a sequence of paths on a curve that is simple and closed so that it constitutes a multi-path cycle E, namely, $\Hcyc E$, and the descriptively good covering is also a path cycle system, $\hSys E$. Then the proof follows from 3$^o$. \\
	
	5$^o$: Observe that if $E \in 2^{H(X)}$  is an element in a degenerate descriptively good covering of $H(X)$,  then it is degenerate descriptively contractible and hence it is descriptively contractible by Corollary~\ref{cor:degenerate}. Then the proof follows from  4$^o$.
\end{proof}

\begin{figure}[!ht]
	\centering
	\subfigure[Vigolo Hawaiian butterfly $\boldsymbol{\Hb E_{t.00}}$ in video frame space $\boldsymbol{\left(frE,\dnear\right)}$ at time $\boldsymbol{t}$ at the beginning of a temporal interval $\boldsymbol{[t,t+05 sec]}$, Betti no. $\boldsymbol{\beta(\Hb E_{t.00}) = 3}$, $\boldsymbol{\Hb E_{t.00}\ \ \dnear\ \Hb E_{t.01}}$]
	{\label{fig:shEBoundary}
		\begin{pspicture}
		(-1.5,-0.5)(4.0,4.0)
		\psframe[linewidth=0.75pt,linearc=0.25,cornersize=absolute,linecolor=blue](-1.45,-0.25)(3.8,3)
		\rput(-0.6,2.7){\footnotesize $\boldsymbol{\left(frE,\dnear\right)}$}
		\psline*[linestyle=solid,linecolor=green!30]%
		(0,0)(1,1)(0,2)(-1,1.5)(-1,0.5)(0,0)
		\psline[linestyle=solid,linecolor=black]%
		(0,0.0)(1,1)(0,2)(-1,1.5)(-1,0.5)(0,0)
		\psline*[linestyle=solid,linecolor=orange!50]%
		(0,0.5)(1,1)(-.55,1.25)(-.55,0.75)(0,0.5)
		\psline[linestyle=solid,linecolor=black]%
		(0,0.5)(1,1)(-.55,1.25)(-.55,0.75)(0,0.5)
		\psdots[dotstyle=o,dotsize=2.2pt,linewidth=1.2pt,linecolor=black,fillcolor=yellow!80]%
		(0,0.5)(1,1)(-.55,1.25)(-.55,0.75)(0,0.5)
		(0,2)(-1,1.5)(-1,0.5)(0,0) 
		\psline*[linestyle=solid,linecolor=green!30]%
		(1,1)(2.0,2.0)(3.0,1.5)(3.0,0.5)(2.0,0.0)(1,1)
		\psline[linestyle=solid,linecolor=black]%
		(1,1)(2.0,0.0)(3.0,0.5)(3.0,1.5)(2.0,2.0)(1,1)
		\psline*[linestyle=solid,linecolor=orange!50]%
		(1,1)(2.55,1.25)(2.55,0.75)(2.0,0.5)(1,1)
		\psline[linestyle=solid,linecolor=black]%
		(1,1)(2.55,1.25)(2.55,0.75)(2.0,0.5)(1,1)
		\psdots[dotstyle=o,dotsize=2.2pt,linewidth=1.2pt,linecolor=black,fillcolor=yellow!80]%
		(1,1)(2.0,0.0)(3.0,0.5)(3.0,1.5)(2.0,2.0)
		(2.55,1.25)(2.55,0.75)(2.0,0.5)
		\psline*[linestyle=solid,linecolor=black!80]%
		(1,1)(0.9,0.8)(0.9,0.5)(1.0,0.3)(1.1,0.5)(1.1,0.8)(1,1)
		\psdots[dotstyle=o,dotsize=2.2pt,linewidth=1.2pt,linecolor=black,fillcolor=yellow!80]%
		(1,1)(0.9,0.8)(0.9,0.5)(1.0,0.3)(1.1,0.5)(1.1,0.8)(1,1)
		\psline[linestyle=solid,linecolor=blue,border=1pt]{*-}(0.8,1.5)(1,1)
		\psline[linestyle=solid,linecolor=blue,border=1pt]{*-}(1.2,1.5)(1,1)
		\psline[linestyle=solid,linecolor=blue,border=1pt]{*-}(0.8,1.5)(1.2,1.5)
		\psdots[dotstyle=o,dotsize=3.5pt,linewidth=1.2pt,linecolor=black,fillcolor=red!80]
		(0.8,1.5)(1.2,1.5)(1.0,1)
		\rput(1.0,2.1){\footnotesize $\boldsymbol{\Hb E_{t.00}}$}
		\rput(-0.8,1.85){\footnotesize $\boldsymbol{cyc Ha}$}
		\rput(-0.25,1.5){\footnotesize $\boldsymbol{cyc Hb}$}
		\rput(2.85,1.85){\footnotesize $\boldsymbol{cyc Ha'}$}
		\rput(2.25,1.5){\footnotesize $\boldsymbol{cyc Hb'}$}
		\rput(1.35,0.95){\footnotesize $\boldsymbol{v_0}$}
		\rput(0.7,1.7){\footnotesize $\boldsymbol{v_1}$}
		\rput(1.3,1.7){\footnotesize $\boldsymbol{v_2}$}
		\end{pspicture}}\hfil
	\subfigure[Vigolo Hawaiian butterfly $\boldsymbol{\Hb E_{t.01}}$ in video frame space $\boldsymbol{\left(frE',\dnear\right)}$ at time $\boldsymbol{t+0.1 sec}$ in temporal interval $\boldsymbol{[t,t+05 sec]}$, Betti no. $\boldsymbol{\beta(\Hb E_{t.01}) = 3}$, $\boldsymbol{\Hb E_{t.00}\ \dnear\ \Hb E_{t.01}}$]
	{\label{fig:shEBoundary2}
		\begin{pspicture}
		(-1.5,-0.5)(4.0,4.0)
		\psframe[linewidth=0.75pt,linearc=0.25,cornersize=absolute,linecolor=blue](-1.55,-0.25)(3.8,3)
		\rput(-0.3,2.7){\footnotesize $\boldsymbol{\left(frE',\dnear\right)}$}
		\psline*[linestyle=solid,linecolor=green!10]%
		(0,0)(1,1)(0,2)(-1,1.5)(-1,0.5)(0,0)
		\psline[linestyle=solid,linecolor=black]%
		(0,0.0)(1,1)(0,2)(-1,1.5)(-1,0.5)(0,0)
		\psline*[linestyle=solid,linecolor=orange!20]%
		(0,0.5)(1,1)(-.55,1.25)(-.55,0.75)(0,0.5)
		\psline[linestyle=solid,linecolor=black]%
		(0,0.5)(1,1)(-.55,1.25)(-.55,0.75)(0,0.5)
		\psdots[dotstyle=o,dotsize=2.2pt,linewidth=1.2pt,linecolor=black,fillcolor=yellow!80]%
		(0,0.5)(1,1)(-.55,1.25)(-.55,0.75)(0,0.5)
		(0,2)(-1,1.5)(-1,0.5)(0,0) 
		\psline*[linestyle=solid,linecolor=green!10]%
		(1,1)(2.0,2.0)(3.0,1.5)(3.0,0.5)(2.0,0.0)(1,1)
		\psline[linestyle=solid,linecolor=black]%
		(1,1)(2.0,0.0)(3.0,0.5)(3.0,1.5)(2.0,2.0)(1,1)
		\psline*[linestyle=solid,linecolor=orange!20]%
		(1,1)(2.55,1.25)(2.55,0.75)(2.0,0.5)(1,1)
		\psline[linestyle=solid,linecolor=black]%
		(1,1)(2.55,1.25)(2.55,0.75)(2.0,0.5)(1,1)
		\psdots[dotstyle=o,dotsize=2.2pt,linewidth=1.2pt,linecolor=black,fillcolor=yellow!80]%
		(1,1)(2.0,0.0)(3.0,0.5)(3.0,1.5)(2.0,2.0)
		(2.55,1.25)(2.55,0.75)(2.0,0.5)
		\psline*[linestyle=solid,linecolor=black!50]%
		(1,1)(0.9,0.8)(0.9,0.5)(1.0,0.3)(1.1,0.5)(1.1,0.8)(1,1)
		\psdots[dotstyle=o,dotsize=2.2pt,linewidth=1.2pt,linecolor=black,fillcolor=yellow!80]%
		(1,1)(0.9,0.8)(0.9,0.5)(1.0,0.3)(1.1,0.5)(1.1,0.8)(1,1)
		\psline[linestyle=solid,linecolor=blue,border=1pt]{*-}(0.8,1.5)(1,1)
		\psline[linestyle=solid,linecolor=blue,border=1pt]{*-}(1.2,1.5)(1,1)
		\psline[linestyle=solid,linecolor=blue,border=1pt]{*-}(0.8,1.5)(1.2,1.5)
		\psdots[dotstyle=o,dotsize=3.5pt,linewidth=1.2pt,linecolor=black,fillcolor=red!80]
		(0.8,1.5)(1.2,1.5)(1.0,1)
		\rput(1.0,2.1){\footnotesize $\boldsymbol{\Hb E_{t.1}}$}
		\rput(-0.8,1.85){\footnotesize $\boldsymbol{cyc Ha}$}
		\rput(-0.25,1.5){\footnotesize $\boldsymbol{cyc Hb}$}
		\rput(2.85,1.85){\footnotesize $\boldsymbol{cyc Ha'}$}
		\rput(2.25,1.5){\footnotesize $\boldsymbol{cyc Hb'}$}
		\rput(1.35,0.95){\footnotesize $\boldsymbol{v_0}$}
		\rput(0.7,1.7){\footnotesize $\boldsymbol{v_1}$}
		\rput(1.3,1.7){\footnotesize $\boldsymbol{v_2}$}
		\end{pspicture}}\hfil
	\caption[]{Persistent butterfly shapes~\cite{Peters2021temporalProximity} in a pair of video frame descriptive proximity spaces}
	\label{fig:frames2}
\end{figure}

\section{Application}
This section briefly introduces an application of descriptively proximal nerves in a topology of data approach to detecting close good covers of video frame shapes that appear, disappear and sometimes reappear in a sequence of video frames.  The basic approach is to track the persistence of descriptively proximal video frame shapes that have homotopic nerve presentations. 

\begin{definition}{\rm [{\bf Path Nerve Presentation}]} $\mbox{}$\\
	Let $\mathcal{B} = \left\{g_1,...\right\}$ be the basis for a free group $G$. Also let $H(X) = \left\{\hcyc E\right\}$, a  collection of path cycles $\hcyc E$ with nonvoid intersection in a planar space $X$.  A {\bf path nerve presentation} is a continuous mapping
	\begin{align*}
	f:H(X)&\to
	\displaystyle
	G\left(v=\sum_{\substack{
			k\in\mathbb{Z}\\ 
			g\in\mathcal{B}}
	}kg: v\in\hcyc E\right)\\
	&\to G(\mathcal{B},+),
	\end{align*}
	from $H(X)$ to a corresponding free group $G$.
\end{definition}

\begin{theorem}{\rm ~\cite{Peters2021KievConf}}\label{theorem:hcycPresentation}
	Every homotopic cycle in a CW space has a free group presentation.
\end{theorem}

Recall that a Betti number is a count of the number generators in a free group~\cite[\S 4,p. 24]{Munkres1984}.

\begin{theorem}{\rm ~\cite{Peters2021temporalProximity}}\label{theorem:nested hCycPresentation}
	Every free group presentation of nested 1-cycles nerve has a Betti number.  
\end{theorem}

The result from Theorem~\ref{theorem:nested hCycPresentation} provides a stepping stone to tracking the persistence of good covers of video frame shapes.  A frame shape persists, provided it continues to appear over a sequence of consecutive video frames. 

\begin{example}
	A pair of descriptively contractible nerves in two video frames, each identified with a descriptive proximity space, is shown in Fig.~\ref{fig:frames2}. For each frame $X$, let the descriptive proximity space $X = H(X)$.  From Theorem~\ref{theorem:desGoodCover}.1$^o$, each frame has a descriptively good cover. In that case, from Theorem~\ref{theorem:proximalJordan}.4$^o$, the $\bdy(\cl_\Phi(H(X)))$  satisfies the Jordan Curve Theorem.
	
	In this example, a nerve is a collection of time-constrained Hawaiian butterfly path cycles (denoted by $\Hb E_{t}$ at time $t$) with nonvoid intersection  such as those in Fig.~\ref{fig:frames2}. Let $\dnear$ be defined in terms of the Betti number of the free groups derived from each nerve, {\em i.e.}, 
	\[
	\Phi(\Hb E_{t}) = \mathcal{B}(\Hb E_{t})
	\]
	Since the free group presentations $\hcyc$ cycles of the Hawaiian butterflies in Fig.~\ref{fig:frames2} have the same Betti number, namely,
	\[
	\mathcal{B}(\Hb E_{t.00})=\mathcal{B}(\Hb E_{t.1}) = 3,
	\]
	then we have
	\[
	\Hb E_{t.00}\ \dnear\ \Hb E_{t.1}.
	\]
	Hence, the persistence of a particular butterfly over a sequence of video frames can be tracked in terms of its Betti number.  In this example, the butterfly represented in Fig.~\ref{fig:frames2} persists for a 10th of a second.  
	\qquad\textcolor{blue}{\Squaresteel}
\end{example}

The motivation for considering free group presentations of polytopes (e.g., nested cycles with nonvoid intersection) covering frame shapes is that we can then describe frame shapes in terms of their Betti numbers.

Frame shapes are approximately descriptively close, provided the difference between the Betti numbers of the free group presentations of the corresponding homotopic nerves, is close.  Determining the persistence of frame shapes then reduces to tracking the appearance, disappearance and possible reappearance of the shapes in terms of their recurring Betti numbers.  For an implementation of this approach to tracking the persistence of polytopes covering brain activation regions in resting state (rs)-fMRI videos, see~\cite{Peters2017fMRInerveStructures}.

\begin{appendix}	
	\section{\v{C}ech Proximity}\label{ap:Cech}
	A nonempty set $X$  equipped with the relation $\delta$ is a  \v{C}ech proximity space (denoted by  $(X,\near$))~\cite[\S 2.5,p 439]{Cech1966}, provided  provided the following axioms are satisfied.\\
	
	\noindent {\bf \v{C}ech Axioms}
	
	\begin{description}
		\item[({\bf P}.0)] All nonempty subsets in $X$ are far from the empty set, {\em i.e.}, $A\ \not{\near}\ \emptyset$ for all $A\subseteq X$.
		\item[({\bf P}.1)] $A\ \near\ B \Rightarrow B\ \near\ A$.
		\item[({\bf P}.2)] $A\ \cap\ B\neq \emptyset \Rightarrow A\ \near\ B$.
		\item[({\bf P}.3)] $A\ \near\ \left(B\cup C\right) \Rightarrow A\ \near\ B$ or $A\ \near\ C$.
	\end{description}

	
	%

	The closure of a subset $A$, denoted by $\cl A$, of the proximity space $X$ is the set of all points in $X$ which are near $A$:
	\[  
	\cl A =\{  x\in X \  : \  x\ \delta_L \ A \}.
	\]
	Note that $A$ is closed, provided $\cl A=A$. \\
	
	\begin{lemma}{\rm ~\cite[p. 9]{Smirnov1952}}
		The closure of any nonempty set $E$ in a proximity space $X$ is the set of all points which are close to $E$.
	\end{lemma}

	We define a nearness relation on $\mathbb{R}$ as follows \cite[\S 1.7, p. 48]{Naimpally2013}. Two nonempty subsets $A$ and $B$ of $\mathbb{R}$ are near if and only if the Hausdorff distance~\cite{Hausdorff1914} $D(A,B)=0$, where
	\[
	D(A,B)=\begin{cases}
	\inf\{ \abs{a-b} \ : \ a\in A \ \mbox{and} \ \  b\in B  \},  &  \mbox{if} \  \ A,B\neq \emptyset, \\ 
	\infty, &  \mbox{if} \ \  A=\emptyset \  \ \mbox{or} \  \ B=\emptyset.
	\end{cases}
	\]
	
	Note that $\mathbb{R}$ is symmetric (or weakly regular), since  $\mathbb{R}$ satisfies the following condition~\cite[\S 3.1, p. 71]{Naimpally2013}.
	\[
	(*) \quad x  \ \mbox{is near} \  \{y\}  \  \Rightarrow \ y  \ \mbox{is near} \  \{x\}. 
	\]
	
	In that case, this nearness relation defines  a Lodato proximity $\delta_L$ on $\mathbb{R}$  by  \cite[\S 3, Theorem 3.1]{Naimpally2013} 
	\[
	A \ \delta_L \ B  \ :\Leftrightarrow \  \cl A \cap  \cl B \neq \emptyset,
	\]
	where $\cl E=\{x\in \mathbb{R} \ : \ D(x,E)=0\}$. \\
	
	The topological space $X$ satisfying $(*)$  becomes a \v{C}ech-Lodato proximity space $(X,\delta_L)$ where $\delta_L$ is defined by
	\[ 
	A \ \delta_L \ B  \ :\Leftrightarrow \  \cl A \cap  \cl B \neq \emptyset, 
	\]
	and $\cl E$ is the closure of $E \subset X$ with respect to the topology on $X$.\\
	
	We assume that the proximity on  the closed interval $[0,1]$ is the subspace  proximity \cite[\S 3.1, p. 74]{Naimpally2013}  induced by the (metric) proximity   on $\mathbb{R}$. 
	
	\section{Descriptive Proximity}\label{app:dnear}
	This section briefly gives an introduction to a framework for the introduction of descriptive proximities~\cite{DiConcilio2016descriptiveProximity} between shapes such as profiles of faces and butterfly wings represented as nonempty sets.  A {\bf description} of a shape is a feature vector $\Phi(E)\in \mathbb{R}^n$ of $n$ real-values whose components are probe function values representing quantifiable shape characteristics. Let $\sh E$ be a planar shape of interest and $f:\sh E\to \mathbb{R}$ is a probe function that returns a real value that quantifies a characteristic of $\sh E$ such as the magnitude and direction of a distinguished vertex, which is a vector field in $\sh E$.
	
\begin{example}\label{ex:featureVector}
Let $HbE_{t.00}$ be the butterfly shape with vertex $v_0$ in Fig.~\ref{fig:frames2}.  In this case, $v_0$ is distinguished, since $\cyc H_a\cap \cyc H_{a'} = v_0$ ({\em i.e.}, $v_0$ is in the intersection of a pair of cycles on the boundaries of the wings of the butterfly shape).  This vertex is an example of a vector field at a known location $(x,y)$, since $v_0$ is a physical picture element with magnitude $\abs{v_0}$ and direction $\theta_{v_0}$ in a video frame, {\em e.g.},
\begin{align*}
\abs{v_0} &= \norm{v_0}.\\
\theta_{v_0} &=\ \mbox{tan}^{-1}\left[\frac{y}{x}\right].
\end{align*}
Then 
\[
\Phi(HbE_{t.00}) = \left(\abs{v_0},\theta_{v_0}\right)\ \mbox{description of butterfly}
\]
provides a basis for checking the descriptive proximity of a pair video frame butterfly shapes.  For example, construct a similar feature vector for butterfly shape $HbE_{t.1}$ that appears in a video frame a 10$^{th}$ of second after the appearance of shape $HbE_{t.00}$ in Fig.~\ref{fig:frames2}.  This pair shapes will be descriptively proximal, provided 
\[
\norm{\Phi(HbE_{t.00})-\Phi(HbE_{t.1})} < \varepsilon,
\]
for some small number.  In  that case, we write
\[
\Phi(HbE_{t.00}\ \dnear\ \Phi(HbE_{t.1},\ \mbox{\em i.e.},
\]
this pair of butterflies is descriptively proximal.
\qquad\textcolor{blue}{\Squaresteel}
\end{example}
	
	let $f(x)\in \vec{v}$ be a characteristic such as the magnitude and direction of the centroid that is a vector field of $\sh E$. Let $x\in\sh E$.  A {\bf probe function} $f:\sh E\to \mathbb{R}$ such that, for each  $x\in\vec{v}$, $f(x)$ quantifies a characteristic of $\sh E$.

	Nonempty sets $A,B\subset X$ with overlapping descriptions are descriptively proximal (denoted by $A\ \dnear\ B$),{\em i.e.},
\begin{align*}
\varepsilon &\in \mathbb{R}^+.\\
A\ &\dnear\ B,\ \mbox{provided}\ \abs{\dnear(A,B)}<\varepsilon.
\end{align*}
	The descriptive intersection~\cite{Peters2013mcsintro} of nonempty subsets in $A\cup B$ (denoted by $A\ \dcap\ B$) is defined by
	\[
	A\ \dcap\ B = \overbrace{\left\{x\in A\cup B: \Phi(x) \in \Phi(A)\ \cap\ \Phi(B)\right\}.}^{\mbox{\textcolor{blue}{\bf {\em i.e.}, $\boldsymbol{\mbox{Descriptions}\ \Phi(A)\ \&\ \Phi(B)\ \mbox{overlap}}$}}}
	\] 
	
	Let $2^X$ denote the collection of all subsets in a nonvoid set $X$. A nonempty set $X$ equipped with the relation $\dnear$ with non-void subsets $A,B,C\in 2^X$ is a descriptive proximity space, provided the following descriptive forms of the \v{C}ech axioms are satisfied.\\
	
	\noindent {\bf  Descriptive Proximity Axioms}
	
	\begin{description}
		\item[({\bf dP}.0)] All nonempty subsets in $2^X$ are descriptively far from the empty set, {\em i.e.}, $A\ \not{\dnear}\ \emptyset$ for all $A\in 2^X$.
		\item[({\bf dP}.1)] $A\ \dnear\ B \Rightarrow B\ \dnear\ A$.
		\item[({\bf dP}.2)] $A\ \dcap\ B\neq \emptyset \Rightarrow A\ \dnear\ B$.
		\item[({\bf dP}.3)] $A\ \dnear\ \left(B\cup C\right) \Rightarrow A\ \dnear\ B$ or $A\ \dnear\ C$.
	\end{description}
	
Recently, a result for the descriptive conjugacy between a pair of dynmamical systems.

\begin{definition}{\rm!\cite[\S 2, p. 388]{PetersVergili2021AGTvortexes}}.\\
Let $\Phi(E)\in \mathbb{R}^n$ be a vector of $n$ real-values that describe a nonempty set $E$.  Two  proximal descriptive continuous maps $f: (X,\delta_{\Phi_1}) \to (X,\delta_{\Phi_1}) $ and $g: (Y,\delta_{\Phi_2}) \to (Y,\delta_{\Phi_2})$ are said to be proximal descriptive conjugates, provided there exists a proximal descriptive isomorphism $h: (X,\delta_{\Phi_1})  \to (Y,\delta_{\Phi_2})$ such that $g\circ h(A) \underset{\mbox{des}}{=} h\circ f(A)$ for any $A \in 2^X$. The function $h$ is called a proximal descriptive conjugacy between $f$ and $g$. 
\end{definition}

{\bf Corollary 2.17}{\rm!\cite[\S 2, p. 391]{PetersVergili2021AGTvortexes}}.\\ 
If there exists a descriptive proximal conjugacy between two descriptive dynamical systems, then they have isomorphic descriptive fixed sets	
\end{appendix}


\bibliographystyle{amsplain}
\bibliography{NSrefs}

\end{document}